\documentclass[10pt, a4paper]{amsart}
\usepackage{charter}
\usepackage{bm}
\usepackage{amsmath}
\usepackage{amssymb}
\usepackage{amsthm}
\usepackage{graphicx}
\usepackage{url}
\usepackage{enumerate}
\usepackage{xcolor}
\usepackage{mathrsfs} 
\usepackage{commath}
\usepackage{tikz}
\usetikzlibrary{automata}

\usepackage{t1enc}
\usepackage{dsfont}

\newtheorem{thm}{Theorem}

\newtheorem{lem}[thm]{Lemma}
\newtheorem{prop}[thm]{Proposition}

\newtheorem{cor}[thm]{Corollary}

\theoremstyle{definition}
\newtheorem{defn}[thm]{Definition}
\newtheorem{rem}[thm]{Remark}
\newtheorem{exmp}[thm]{Example}

\newcommand{\D}{D}

\renewcommand{\phi}{\varphi}

\DeclareMathOperator{\M}{\mathcal{M}}

\DeclareMathOperator{\Ms}{\mathcal{M}_{\mathit \sigma} }
\DeclareMathOperator{\Mss}{\mathcal{M}_{\mathit \sigma\times \sigma} }

\DeclareMathOperator{\Mse}{\M_{\mathit \sigma}^e}
\DeclareMathOperator{\Msse}{\M_{\mathit \sigma\times\sigma}^e}

\newcommand{\FS}{\alf^\infty}
\newcommand{\uin}{^{(n)}}
\newcommand{\ijn}{^{(n)}_j}
\newcommand{\BB}{\mathscr{B}}

\newcommand{\F}{\mathscr{F}}
\newcommand{\FB}{\F_\BB}
\newcommand{\FBk}{\F_{\BB|k}}

\newcommand{\Bl}{\mathcal{B}}

\newcommand{\alf}{\mathscr{A}}

\DeclareMathOperator{\lang}{\Bl}

\DeclareMathOperator{\CL}{CL}
\DeclareMathOperator{\dirac}{\delta}

\newcommand{\dbar}{{\bar d}}
\newcommand{\dbarm}{\bar{d}_{\mathcal{M}}}
\newcommand{\dund}{{\underline d}}

\newcommand{\Hdbar}{{\bar d}^H}
\newcommand{\Hdbarm}{\bar{d}^H_{\mathcal{M}}}
\newcommand{\Hdund}{{\underline d}^H}

\newcommand{\dHam}{d_\textrm{Ham}}

\newcommand{\weakst}{weak$^*$}

\newcommand{\eps}{\varepsilon}
\newcommand{\R}{\mathbb{R}}
\newcommand{\Z}{\mathbb{Z}}
\newcommand{\N}{\mathbb{N}}

\author{Jakub Konieczny \and Michal Kupsa \and Dominik Kwietniak}

\address[D.\ Kwietniak]{Faculty of Mathematics and Computer Science, Jagiellonian University in Krakow, ul. \L o\-jasiewicza 6, 30-348 Krak\'ow, Poland}\email{dominik.kwietniak@uj.edu.pl}
\urladdr{www.im.uj.edu.pl/DominikKwietniak/}

\address[M.\ Kupsa]{ The Czech Academy of Sciences, Institute of Information Theory and Automation, Prague 8, CZ-18208}\email{kupsa@utia.cas.cz}

\address[J.\ Konieczny]{Camille Jordan Institute,
Claude Bernard University Lyon 1,
43 Boulevard du 11 novembre 1918,
69622 Villeurbanne Cedex, France}
\email{jakub.konieczny@gmail.com}

\title[On $\dbar$-approachability\ldots]{On $\dbar$-approachability, entropy density and $\BB$-free shifts}
\date{\today}

\begin{document}

\begin{abstract}
We study approximation schemes for shift spaces over a finite alphabet using (pseudo)metrics connected to Ornstein's $\dbar$ metric.
This leads to a class of shift spaces we call $\dbar$-approachable. A shift space $\dbar$-approachable when its canonical sequence of Markov approximations converges to it also in the $\dbar$ sense. We give a topological characterisation of chain mixing $\dbar$-approachable shift spaces. As an application we provide a new criterion for entropy density of ergodic measures.
Entropy-density of a shift space means
that every invariant measure $\mu$ of such a shift space is the weak$^*$ limit of a sequence $\mu_n$ of ergodic measures with the corresponding sequence of entropies $h(\mu_n)$ converging to $h(\mu)$.  We prove ergodic measures are entropy-dense for every shift space that can be approximated in the $\dbar$ pseudometric by a sequence of transitive sofic shifts. This criterion can be applied to many examples that were out of the reach of previously known techniques including hereditary $\BB$-free shifts and some minimal or proximal systems. The class of symbolic dynamical systems covered by our results includes also shift spaces where entropy density was established previously using the (almost) specification property.
\end{abstract}
\maketitle


We study approximation schemes for shift spaces over a finite alphabet $\alf$ (subshifts of the full shift $\FS$ over $\alf$). For every shift space $X$ there exists a canonically defined sequence of shifts of finite type (topological Markov approximations) converging to $X$ in a natural topology on the space of all subshifts of $\FS$. This fact, however, is of little practical use, because dynamical properties usually do not convey from a sequence of shift spaces to its limit. Here, we consider another, stronger  topology on the powerset of $\FS$ which is induced by the $\dbar$-pseudometric on $\FS$ or one of its relatives. The pseudometric $\dbar$ is given for $x=(x_j)_{j=0}^\infty,y=(y_j)_{j=0}^\infty\in \FS$ by
\[
\dbar(x,y)=\limsup_{n\to\infty}\frac{1}{n}|\{0\le j <n : x_j\neq y_j\}|.
\]
It is an analog of Ornstein's  metric $\dbarm$ on $\alf$-valued stationary stochastic processes. Since $\dbar$ is bounded, it induces a Hausdorff pseudometric $\Hdbar$ on the space of all nonempty subsets of $\FS$. We examine the existence of an approximating sequence in the $\Hdbar$ sense and study its consequences.
In particular, we study shift spaces we call \emph{$\dbar$-approachable}, which are $\Hdbar$-limits of their own topological Markov approximations. We provide a topological characterisation of chain mixing $\dbar$-approachable shift spaces using the $\dbar$-shadowing property. This can be considered as an analogue for Friedman and Ornstein's characterisation of Bernoulli processes as totally ergodic $\dbarm$-limits of their own Markov approximations \cite{FO70}. We prove that many specification properties imply chain mixing and $\dbar$-approachability. It follows that all shift spaces with the almost specification property (this class includes all $\beta$-shifts and all mixing sofic shifts, in particular all mixing shifts of finite type) are $\dbar$-approachable. In a forthcoming paper \cite{KKK3}, we also construct minimal and proximal examples of $\dbar$-approachable shift spaces, thus proving $\dbar$-approachability is a more general phenomenon than specification.
Regarding consequences of $\Hdbar$ approximation, we show that it implies convergence of simplices of invariant measures in the Hausdorff metric $\Hdbarm$ induced by Ornstein's $\dbarm$ metric on the space $\Ms(\FS)$ of all shift invariant measures on $\FS$.
To be more precise: Write $\Ms(X)$ for the space of all Borel invariant probability measures on a shift space $X$ and $\Mse(X)$ for the set of ergodic measures in $\Ms(X)$. We prove that if a sequence of shift spaces $(X_n)_{n=1}^\infty$ converges to $X$ in $\Hdbar$, then $(\Ms(X_n))_{n=1}^\infty$ and $(\Mse(X_n))_{n=1}^\infty$ converge in $\Hdbarm$ to $\Ms(X)$, respectively to $\Mse(X)$. This allows us to show that certain geometric features of the simplices of invariant measures of shift spaces in the approximating sequence are inherited by the simplex of the limit. The geometry of the space $\Ms(X)$ is, in turn, an important feature of a shift space $X\subseteq\FS$, and besides being interesting per se, it has profound connections with dynamical properties of $X$.
It turns out that for many important shift spaces and, more generally, dynamical systems and continuous time flows on compact metric spaces, $\Mse(X)$ is a nontrivial dense subset of $\Ms(X)$ endowed with the weak$^*$ topology, see \cite{GK,LO,Parthasarathy,Sigmund70,Ville}.
A stronger property is entropy density of ergodic measures introduced by Orey in 1986 \cite{Orey} and F\"ollmer and Orey in 1988 \cite{FO}. A shift space $X$ has entropy density of ergodic measures if for every $\nu\in \Ms(X)$ there exists a sequence $\mu_n \in \Mse(X)$ such that $\mu_n \rightarrow \nu$  in the weak$^*$ topology and, in addition, the measure-theoretic entropies also converge, that is, $h(\mu_n) \rightarrow h(\nu)$ as $n\to\infty$.
There are examples of shift spaces with dense, but not entropy-dense sets of ergodic measures (see \cite{GK}). In addition to being of interest in their own right, density of ergodic measures and entropy-density are strongly related to various results on large deviations and multifractal analysis \cite{Comman,EKW,PS,PS2}. We refer the reader to Comman's article \cite{Comman2017} and references therein for more information about that connection and more large deviations results.

We show that entropy-density of ergodic measures is a $\Hdbar$-closed property. This allows for a new method of proof of entropy density. So far, specification-like conditions were invoked to prove entropy-density (see \cite{EKW,PS}),  and our results extends that to $\Hdbar$ limits of these shift spaces. In particular, mixing shifts of finite type are entropy-dense, hence our result shows that all $\dbar$-approachable chain-mixing shift spaces are also entropy dense. This leads to a new proof of entropy density of systems with almost specification, which is now a direct consequence of entropy-density of mixing shifts of finite type and $\dbar$-approachability. Since we know there are examples of minimal or proximal $\dbar$-approachable shifts, we see that our technique yields entropy-density for examples which were beyond the reach of methods based on specification properties. Furthermore, we prove that there are cases where $\dbar$-approachability does not necessarily hold, but it is still possible to approximate in $\Hdbar$ with a sequence of entropy-dense shift spaces. We apply this technique to hereditary closures of $\BB$-free shifts (a class including many interesting $\BB$-free shifts, see Corollary  \ref{cor:B-free} below). These shift spaces are not weakly mixing, hence they do not have $\dbar$-shadowing property (we show examples that are not $\dbar$-approachable, see Example \ref{exmp:bfree-non-dbar-approach}
). Nevertheless, using the Davenport-Erd\H{o}s theorem they are easily seen to be approximated by naturally defined sequences of transitive sofic shifts, and this implies entropy-density. Note that $\BB$-free shifts do not have even the weakest version of the specification property (see \cite{KLO} for a list of possibilities).
Recall that an integer is \emph{$\BB$-free} if it has no factor in a given set $\BB\subseteq\N$. For example, the set of $\BB_{\text{sq}}$-free integers where $\BB_{\text{sq}}=\{p^2:p\text{ prime}\}$ is just the set of \emph{square-free integers}. Sets of $\BB$-free integers were studied by Chowla, Davenport, and Erd\H{o}s, to name a few. Recently, Sarnak \cite{Sarnak}  initiated the study of dynamics of a shift space $X_\text{sq}$ associated to square-free integers. El Abdalauoi, Lema\'{n}czyk, and de la Rue \cite{ALR} and Dymek et al \cite{DKKPL}
extended Sarnak's approach to $\BB$-free integers for an arbitrary $\BB$.
These systems were also investigated by Avdeeva and her co-authors \cite{Avdeeva, ACS}, Cellarosi and Sinai \cite{CS}, Ku{\l}aga-Przymus and her co-authors \cite{KPL,KPLW,KPLW2}, Peckner \cite{Peckner}, and Keller and his co-authors \cite{KKLem, Keller, Keller-Studia}. Higher dimensional analogs of $\BB$-free shifts also attracted attention, see \cite{BH15,Bartnicka,CV}).

Furthermore, it is known that ergodic measures are dense for hereditary $\BB$-free shifts by \cite{KPLW2}, but a proof of this fact is a consequence of a non-trivial reasoning presented in \cite{KPLW} and \cite{DKKPL} (more precisely, in \cite{KPLW} the analysis of invariant measures was performed under some additional technical assumptions on the set $\BB$ and it was later extended to the general case in \cite{DKKPL}). We extend the main result of \cite{KPLW2} in two directions: we add entropy to the picture and broaden the class of shift spaces for which this result holds (see Remark \ref{rem:more-than-B-free}). Our proof is independent of \cite{KPLW,KPLW2}, and \cite{DKKPL}. As a matter of fact our methods lead to new, often shorter proofs, of many results from \cite{KPLW,KPLW2}, and \cite{DKKPL}. Furthermore, our methods can be adapted to shift spaces with an action of $\Z^d$ for $d\ge 2$ or a countable amenable residually finite group. This will be a subject of our future study (forthcoming paper).

The idea of using $\dbarm$ approximation was independently considered by Dan Thompson \cite{Thompson-note}, who used it in the setting of \cite{CT}.

This paper is organised as follows: in the next section we recall some notation and basic definitions.
In Section \ref{sec:dbar} we discuss Hausdorff (pseudo)metrics $\Hdbar$ and $\Hdbarm$ induced by $\dbar$ and $\dbarm$.
Section \ref{sec:approach} contains the definition and characterisation of $\dbar$-approachability.
In Section \ref{sec:main} we study consequences of $\Hdbar$ convergence of shift spaces and $\Hdbarm$ convergence of simplices of invariant measures. The last three sections contain applications and examples illustrating our approach. In Section \ref{sec:spec} we show that our results contain previous results on entropy-density of shift spaces with a variant of the specification property. Section \ref{sec:bfree} discusses approximation schemes and entropy density of $\BB$-free shifts and their hereditary closures.
In Section \ref{sec:inner} we show that topologically mixing $S$-gap shifts are another example of a shift space allowing for a natural $\Hdbar$-approximation by a sequence of mixing shifts of finite type.

\section{Definitions}\label{sec:definitions}

\subsection{General}
We let $\N$ denote the set of positive integers and $\N_0=\N\cup\{0\}$.
\subsection{Shift spaces and languages}
Let $\alf$ be a finite set, henceforth referred to as the \emph{alphabet}. The \emph{full shift} $\alf^{\infty}$ is the set of all $\alf$-valued infinite sequences indexed by nonnegative integers. We endow $\FS$ with the product topology coming from the discrete topology on $\alf$. A metric on $\FS$ compatible with the product topology is
\[
\rho(x,y)=\begin{cases}
            0, & \mbox{if } x=y, \\
            2^{-\min\{j:x_j\neq y_j\}}, & \mbox{otherwise}.
          \end{cases}
\]
We consider $\FS$ a dynamical system under the action of the \emph{shift} transformation $\sigma\colon\alf^\infty\to\alf^\infty$, where $\sigma(x)_j=x_{j+1}$ for every $x=(x_i)_{i=0}^\infty\in\alf^\infty$ and $j\ge 0$. A \emph{shift space} over $\alf$ is a nonempty, closed, and shift invariant subset of $\alf^\infty$. An \emph{block} or a \emph{word}  over $\alf$ is a finite sequence of symbols from $\alf$. The \emph{length} of a word $w$ over $\alf$ is denoted by $|w|$. The \emph{concatenation} of words $u$ and $v$ is denoted simply as $uv$. We agree that the \emph{empty word} has length $0$. Given $x \in X$ and $i,j\in\N_0$ with $i<j$ we set $x_{[i,j)}$ to be the word
$w=w_1\ldots w_n\in\alf^n$ such that $n=j-i$ and  $x_{i+k-1} = w_k$ for all $1 \leq k \leq n$.
We say that a word $\omega \in  \alf^n$ \emph{appears} in a shift space $X \subseteq \FS$ if there exists some $x\in X$ and $\ell\in\N_0$ such that
$\omega=x_{[\ell,\ell+n)}$. For $n\in\N$, we write $\lang_n(X) \subseteq \alf^n$ for the set of blocks of length $n$ appearing in $X$. The \emph{language} of a shift space $X\subseteq\FS$ is the set $\lang(X)$ of finite words over $\alf$ that appear in $X$. A shift space $X$ is \emph{transitive} if for every words $u,w\in\lang(X)$ there is a word $v$ such that $uvw\in\lang(X)$. We say that a shift space $X$ is \emph{topologically mixing} if for every words $u,w\in\lang(X)$ there is $N\in\N_0$ such that for every $n\ge N$ there is a word $v$ satisfying $uvw\in\lang(X)$ and $|v|=n$.

\subsection{Shifts of finite type}
Recall that every family $\mathscr{F}$ of finite words over $\alf$ determines the collection $X_{\mathscr{F}}$ consisting of such sequences $x=(x_i)_{i=0}^\infty\in\alf^\infty$ that no word from $\mathscr{F}$ appears in $x$ is either empty or a shift space. Conversely, for every shift space $X$ there is a family $\mathscr{F}$ of finite words such that $X=X_\mathscr{F}$. A shift space $X$ is a \emph{shift of finite type} if there exists a finite set $\mathscr{F}$ such that $X=X_{\mathscr{F}}$.

\subsection{Ergodic properties of shift spaces}
The set of all Borel probability measures on a shift space $X$ is denoted by $\M(X)$.
We equip $\M(X)$ with the \weakst\ topology, which is known to be metrizable and compact.
The set of shift invariant probability measures supported on a shift space $X\subseteq\FS$ is denoted by $\Ms(X)$ and $\Mse(X)$ stands for the set of ergodic $\sigma$-invariant measures on $X$.

A point $x\in \FS$ is \emph{generic} for an ergodic measure $\mu\in\Mse(\FS)$,  if for every continuous function $f\colon\FS\to\R$ the sequence
\[
\frac{1}{N}\sum_{j=0}^{N-1}f(\sigma^n(x))
\]
converges as $N\to\infty$ to $\int_{\FS}f\,\text{d}\mu$. Every ergodic measure has a generic point.
Given $\mu\in\Ms(\FS)$, we let $h(\mu)$ denote the Kolmogorov-Sinai entropy of $\mu$. 

	The ergodic measures are \emph{entropy dense} if
	for every measure $\mu\in \Ms(X)$, every neighborhood $U$ of $\mu$ in $\Ms(X)$ and every $\eps>0$ there is an ergodic measure $\nu\in U$ with
	$|h(\nu)-h(\mu)|<\eps$.

        Let us point out that property of having entropy-dense ergodic measures is preserved by conjugacy. In other words, two conjugated dynamical systems either both have entropy dense ergodic measures or both do not have this property.

\subsection{Hereditary shifts}
Suppose that $\alf$ is additionally equipped with a (total) order $\leq$. In particular, if $\alf=\{0,1\}$ then $\leq$ is the usual order. We equip $\FS$ with coordinate-wise partial order also denoted by $\leq$. This means that for $x,y \in \FS$ we have $x \leq y$ if and only if $x_i \leq y_i$ for all $i \in \N_0$. A shift space $X\subseteq\FS$ is \emph{hereditary} if for every $x\in X$ and $y\in\FS$ with $y \leq x$ we have $y \in X$. The hereditary closure $\tilde X$ of $X$ is the smallest hereditary shift containing $X$ and consists of those $y \in \FS$ for which there exists $x \in X$ with $y \leq x$. For more details, see \cite{K}.

\subsection{Sofic shifts}
An (oriented) $\alf$-labelled (multi)graph $G = (V,E,\tau)$ consists of a vertex set $V$, an edge set $E$ and a label map $\tau \colon E \to \alf$. Each edge $e \in E$ has two endpoints $i(e), t(e) \in V$, sometimes called, respectively, the initial vertex and the terminal vertex of $e$. A path of length $\ell$ (finite or infinite) in $G$ is a sequence of $\ell$ edges $e_1, e_2, \dots$ such that for each $i < \ell$, the terminal vertex of $e_i$ is the same as the initial vertex of $e_{i+1}$, equivalently,  a sequence of edges $e_1, e_2, \dots$ is a path if there exists a sequence of vertices $v_1,v_2,\dots \in V$ such that $i(e_i) = v_i$ and $t(e_i) = v_{i+1}$ for every $i<\ell$.
	
The shift $X_G \subseteq \FS$ corresponding to $G$ consists of all sequences $x \in \FS$ such that there exists an infinite path $e_1,e_2, \dots$ on $G$ such that $x_i = \tau(e_{i+1})$ for each $i \in \N_0$. That is, the shift space $X_G$ is obtained by reading off labels of all
infinite paths on $G$. A shift space is a \emph{sofic shift} if there exists a labelled graph $G= (V,E,\tau)$ with $V$ finite such that $X=X_G$. Then we also say that $X$ is presented by $G$. Every shift of finite type is sofic. A sofic shift is transitive if and only if it can be presented by a (strongly) connected graph (each pair of vertices can be connected by a path), see \cite[Prop. 3.3.11]{LM}. A sofic shift is topologically mixing if and only if it can be presented by a (strongly) connected graph such that there are two closed paths on $G$ of coprime lengths.

\subsection{Markov approximations}
We recall the topological approximation scheme of shift spaces by canonically defined sequences of shifts of finite type (see \cite{ChRU}, \cite[p. 111]{DGS}, or \cite{Kurka} for more details).

Given a shift space $X\subseteq\FS$ and a family $\mathscr{F}$ of finite words over $\alf$ such that $X=X_\mathscr{F}$ one sets $\mathscr{F}[n]$ to be the set of all words $w$ in $\mathscr{F}$ with the length $|w|\le n+1$. Clearly, the shift space $X^M_n=X_{\mathscr{F}[n]}$ is a shift of finite type whose language has the same words of length at most $n+1$ as $X$. The shift space $X^M_n$ is called the $n$-th \emph{(topological) Markov approximation} of $X$ or \emph{finite type approximation of order $n$ to $X$}.

Alternative description of $X^M_n$ uses Rauzy graphs.  The \emph{$n$-th Rauzy graph} of $X$ is a labelled graph $G_n=(V_n,E_n,\tau_n)$,
where $V_n=\lang_n(X)$, $E_n=\lang_{n+1}(X)$, and for $w=w_0w_1\ldots w_{n}\in E_n$ we set the initial vertex
$i(w)$ of $w$ to be $w_0\ldots w_{n-1}\in V_n$, the terminal vertex $t(w)$ of $w$ is $w_1\ldots w_{n}\in V_n$,
and the label $\tau_n(w)=w_0\in\alf$. The shift space $X_n$ presented by $G_n$ is clearly sofic, and by Proposition 3.62 in \cite{Kurka}, it satisfies
$\lang_{j}(X_n)=\lang_{j}(X)$ for $j=1,\ldots,n+1$. It is now easy to see that $X_n$ is the $n$-th  topological Markov approximations for $X$.

\section{Alternative (pseudo)metrics for $\FS$ and $\Ms(\FS)$ and their hyperspaces}\label{sec:dbar}

\subsection{The functions $\dund$, $\dbar$, and $\dbarm$}
We now discuss pre- and pseudometric $\dund$ and $\dbar$ for the space $\FS$ and Ornstein's metric $\dbarm$ on $\Ms(\FS)$. These functions are not compatible with the natural topologies on $\FS$ and $\Ms(\FS)$, which are the product topology and the weak$^*$ topology, respectively. Note that usually $\dbarm$ is also denoted by $\dbar$, but as we will juggle between $\dbar$ and $\dbarm$ a lot, we decided to change the notation to avoid confusion.  We call $\dbarm$ the `d-bar distance for measures' and its pseudometric version for $\FS$, we call `pointwise d-bar'.

Given $x=(x_n)_{n=0}^\infty,y=(y_n)_{n=0}^\infty\in\FS$ we define
\begin{gather*}
  \dbar(x,y)=\limsup_{n\to\infty}\frac{1}{n}|\{0\le j < n:x_j\neq y_j\}|, \\
  \dund(x,y)=\liminf_{n\to\infty}\frac{1}{n}|\{0\le j < n:x_j\neq y_j\}|.
\end{gather*}
The function $\dund$ is only a \emph{premetric}, that is, $\dund$ is a real-valued, nonnegative, symmetric function on $\FS\times\FS$ vanishing on the diagonal $\{(x,y)\in \FS\times\FS: x=y\}$. It is easy to see that the triangle inequality fails for $\dund$ on $\FS$. The function $\dbar$ is a \emph{pseudometric}, that is, $\dbar$ is a premetric satisfying the triangle inequality and the implication $\dbar(x,y)=0\implies x=y$ fails, so $\dbar$ is not a metric on $\FS$ if $\alf$ has at least two elements.
Furthermore, $\dbar,\dund\colon\FS\times\FS\to [0,1]$ are both shift invariant, that is, $\dund(x,y)=\dund(\sigma(x),\sigma(y))$ and $\dbar(x,y)=\dbar(\sigma(x),\sigma(y))$ for all $x,y\in\FS$.

The distance function $\dbarm$ is defined using joinings. We say that a measure $\lambda$ on $\FS\times\FS$ is a \emph{joining} of shift invariant measures $\mu$ and $\nu$, if $\lambda$ is $\sigma\times\sigma$-invariant and $\mu$ and $\nu$ are the marginal measures for $\lambda$ under the projection to the first, respectively the second, coordinate. By $J(\mu,\nu)$ we denote the set of all joinings of $\mu$ and $\nu$. It is a non-empty set because the product measure $\mu\times\nu$ is always a joining of $\mu$ and $\nu$. We can now define $\dbarm$ distance on $\Ms(\FS)$ as follows:
\begin{align*}
  \dbarm(\mu,\nu)=\inf_{\lambda\in J(\mu,\nu)}\int_{\FS\times\FS} d_0(x,y) \dif
   \lambda,
\end{align*}
where $d_0(x,y)=1$ if $x_0\neq y_0$ and $d_0(x,y)=0$ otherwise.
It is well-known that the $\dbarm$ is a metric on $\Ms(\FS)$ and the convergence in this metric implies weak$^*$ convergence.
 This metric was introduced by Don Ornstein (for more details, see \cite{Ornstein}).
Furthermore, the space $\Ms(\FS)$ under the $\dbarm$-metric is complete but not separable, in particular the space $\Ms(\FS)$ endowed with $\dbarm$ is not compact. The space $\Mse(\FS)\subseteq\Ms(\FS)$ of ergodic measures is $\dbarm$-closed, as are the spaces of strong-mixing and Bernoulli measures on $\FS$.
Entropy function $\mu\mapsto h(\mu)$ is continuous on $\Ms(\FS)$ under $\dbarm$.

\subsection{Hausdorffication of $\dund$, $\dbar$, and $\dbarm$}
In order to introduce the notion of approximation of one shift with another, we extend the functions $\dbarm$, $\dbar$ and $\dund$ so that they make sense for pairs of subsets of $\FS$ and $\Ms(\FS)$ following the well-known construction of the Hausdorff metric.

In general, a Hausdorff premetric may be defined on the powerset of any bounded premetric (or pseudometric, or metric) space. If $(Z,\rho)$ is a set equipped with a bounded premetric, then we define the \emph{Hausdorff premetric} $\rho^H$ induced by $\rho$ on the space of all nonempty subsets of $Z$. 
For a point $z\in Z$ and nonempty $A,B\subseteq Z$, we put
\[\rho(z,B)=\inf_{b\in B}\rho(z,b),\quad\text{and}\quad \rho^H(A,B)=\max\left\{\sup_{a\in A}\rho(a,B),\ \sup_{b\in B}\rho(b,A)\right\}.\]
The function $\rho^H$ is a premetric, which becomes a pseudometric whenever $Z$ is a bounded pseudometric space. Even if $\rho$ is a bounded metric on $Z$, its Hausdorff counterpart induced on the powerset of $Z$ might still be only a pseudometric. This is because for nonempty $A,B\subset Z$ we have $\rho(A,B)=\rho(\overline{A},\overline{B})$, where $\overline{(\cdot)} $ is the closure operator naturally defined on the powerset of $(Z,\rho)$. Nevertheless, if $\rho$ is a bounded metric, then the Hausdorff pseudometric $\rho^H$ induced by $\rho$ becomes a metric when restricted to the set $\CL(Z,\rho)$ of closed nonempty subsets of $(Z,\rho)$.
We apply these ideas to the metric space $(\Ms(\FS),\dbarm)$, the pseudometric space  $(\FS,\dbar)$,  and the premetric space $(\FS,\dund)$ obtaining a metric $\Hdbarm$ on the space $\CL(\Ms(\FS),\dbarm)$ of nonempty closed subsets of $(\Ms(\FS),\dbarm)$ together with a pseudometric $\Hdbar$ and a premetric $\Hdund$ on the set of all nonempty subsets of $\FS$.

Below, we discuss some properties of the convergence of sets with respect to the Hausdorff metric $\rho^H$ in case when $(Z,\rho)$ is not necessarily compact bounded pseudo\-metric space. In this setting some properties, well-known in the compact case, fail. For example, equivalent metrics, that is metrics $\rho,\tilde\rho$ inducing the same topology on $Z$ may induce non-homeomorphic spaces $(\CL(Z),\rho^H)$ and $(\CL(Z),\tilde{\rho}^H)$.

The following propositions gathers some properties of the Hausdorff metric. 
Proofs can be found in the literature (see \cite[\S2.15]{IN}).

\begin{prop} \label{prop:Lim-properties} Let $(Z,\rho)$ be a bounded metric space. 
If $(Z,\rho)$ is a complete metric space, then so is $(\CL(Z),\rho^H)$.
\end{prop}

It follows from Proposition \ref{prop:Lim-properties} that the Hausdorff metric $\Hdbarm$ induced for $\CL(\Ms(\FS))$ by $\dbarm$ is complete.
Therefore the usual Cauchy condition provides a criterion for convergence, 
but even if we know that $(\Ms(X_k))_{k=1}^\infty$ converges in $\dbarm$ to some $\M\in\CL(\Ms(\FS),\dbarm)$, it is not clear if there exists a shift space $X\subseteq\FS$ such that $\M=\Ms(X)$. Nevertheless, if $X_1\supseteq X_2 \supseteq \ldots$, then we can set $X$ to be the intersection of all $X_n$'s. 


\begin{prop}\label{prop:intersection}
If $(X_n)_{n=1}^\infty$ is a decreasing sequence of shift spaces over $\alf$, then
$X=\bigcap_{n=1}^\infty X_n$ 
is a nonempty shift space such that
$\Ms(X)=\bigcap_{n=1}^\infty\Ms(X_n)$.
\end{prop}
\begin{proof}
Since $X_1\supseteq X_2\supseteq\ldots\supseteq X$, we have
\[
\Ms(X)=\Ms(\bigcap_{n=1}^\infty X_n)\subseteq\bigcap_{n=1}^\infty\Ms(X_n) .\]
On the other hand,
if $\mu \in \bigcap_{n=1}^\infty\Ms(X_n)$, then $\mu(X_n)=1$ for every $n\ge 1$. Therefore $\mu(X)=1$.
It follows that $\mu\in\Ms(X)$. 
\end{proof}

\begin{lem}\label{gen-scheme-Cauchy}
If $(X_n)_{n=1}^\infty$ is a decreasing sequence of shift spaces over $\alf$ such that the sequence $(\Ms(X_n))_{n=1}^\infty$ is $\Hdbarm$-Cauchy, then
 $X=\bigcap_{n=1}^\infty X_n$ is a shift space such that $\Hdbarm(\Ms(X_n),\Ms(X)) \to 0$ as $n \to \infty$.
\end{lem}
\begin{proof}
It follows from Proposition \ref{prop:Lim-properties} that $(\Ms(X_k))_{k=1}^\infty$ converges to some $\M\in\CL(\FS,\dbarm)$ with respect to $\Hdbarm$.
Using Proposition \ref{prop:intersection} we can easily identify this limit with $\Ms(X)$.
\end{proof}

\section{$\dbar$-approachability}\label{sec:approach}

In this section, we introduce \emph{$\dbar$-approachable shift spaces} which are limits of their own canonical finite type approximations not only in the `usual' Hausdorff metric topology, but also in the topology given by $\Hdbar$-pseudo\-metric.

\begin{defn}
We say that a shift space $X \subseteq \FS$ is \emph{$\dbar$-approachable} 
if 
the sequence $X^M_1,X^M_2,\ldots$ of its Markov approximations of $X$
satisfies $\Hdbar(X^M_n,X)\to 0$ as $n\to\infty$.
\end{defn}
Clearly, every shift of finite type is $\dbar$-approachable. More examples of $\dbar$-approachable shift spaces will follow from the characterisation of  chain-mixing $\dbar$-approachable shift spaces by a notion we call \emph{$\dbar$-shadowing}. The $\dbar$-shadowing property is closely related to the \emph{average shadowing property} introduced by Blank \cite{Blank} and discussed in more detail in \cite{KKO}. Actually, following the ideas presented in \cite{KKO}, one can prove that a shift space $X$ with the average shadowing property also has $\dbar$-shadowing. Furthermore, for a surjective shift space $X$ (that is, when $\sigma(X)=X$), $\dbar$-shadowing implies the average shadowing property. We leave the details to an interested reader. We decided to use $\dbar$-shadowing because it is much easier to use for symbolic system\footnote{We will apply similar strategy to several other notions: instead of presenting a general (that is, stated for continuous maps acting on compact metric spaces) definition of some properties (e.g. chain transitivity, chain mixing, specification and its variants), we will present equivalent definitions adapted to the symbolic dynamics setting.} and allows us to keep our paper self-contained by avoiding a very technical detour to topological dynamics.

\begin{defn}
A shift space $X$ has the \emph{$\dbar$-shadowing property} if for every $\eps>0$ there is $N\in\N$ such that for every sequence $(w^{(j)})_{j=1}^\infty$
of words in $\lang(X)$ satisfying $|w^{(j)}|\ge N$ for $j=1,2,\ldots$ there is $x'\in X$ such that setting $x=w^{(1)}w^{(2)}w^{(3)}\ldots$ we have
$\dbar(x,x')<\eps$.
\end{defn}

It is easy to see that $\dbar$-shadowing implies $\dbar$-approachability, but to prove the converse we need an additional assumption. We say that a shift space is \emph{chain transitive} (respectively, \emph{chain mixing})
if $X^M_n$ is transitive (respectively, topologically mixing) for all except finitely many $n$'s. For shift spaces this definition is equivalent to the usual one phrased in terms of $\delta$-chains (see \cite{Kurka}).

Our main result in this section characterises chain mixing $\dbar$-approachable shift spaces.
As mentioned above, this characterisation is  a topological counterpart of the result saying that a totally ergodic shift invariant probability measure is Bernoulli
if and only if its canonical $k$-step Markov approximations converge to $\mu$ with respect to $\dbarm$.

\begin{thm}\label{thm:characterisation}
For a shift space $X\subseteq\FS$ the following conditions are equivalent:
\begin{enumerate}
  \item \label{cond:i} There exists a descending sequence $(X_n)_{n=1}^\infty$ of mixing sofic shifts  such that $X=\bigcap_{n=1}^\infty X_n$ and $\Hdbar(X,X_n)\to 0$ as $n\to\infty$,
  \item \label{cond:ii} $\sigma(X)=X$ and $X$ has the $\dbar$-shadowing property,
  \item \label{cond:iii} $X$ is chain mixing and $\dbar$-approachable. 
\end{enumerate}
\end{thm}

As a consequence of Theorem \ref{thm:characterisation} we obtain that a transitive, but not topologically mixing shift of finite type is an example of a $\dbar$-approachable shift space which is not chain mixing, hence it does not have the $\dbar$-shadowing property. In Example \ref{exmp:non-dbar-approach} we show that there are non-$\dbar$-approachable transitive sofic shifts. It follows that condition \ref{cond:i} in Theorem \ref{thm:characterisation} can not be relaxed.

The proof of the following lemma is straightforward. We will later generalise it in Section \ref{sec:spec}, see Lemma \ref{lem:dbar-shadow-almost-spec} there.
\begin{lem}\label{lem:spec-dbar-shadowing}
Every mixing sofic shift space has the $\dbar$-shadowing property.
\end{lem}
\begin{proof}
Recall that the condition saying that $X$ is a mixing sofic shift space implies that there exists $k\ge 0$ such that for every two words $u,w\in\lang(X)$ one can find\footnote{This property is called the \emph{specification property} and we discuss it in more detail in Section \ref{sec:spec}.} a word $v$ of length $k$ such that $uvw\in\lang(X)$. Fix $\eps>0$. Let $N$ be such that  $k/N<\eps$. Let $(w^{(j)})_{j=1}^\infty$ be a sequence
of words in $\lang(X)$ satisfying $|w^{(j)}|\ge N$ for $j=1,2,\ldots$. For $j=1,2,\ldots$, we set $\ell(j)=|w^{(j)}|$ and let $u^{(j)}=w^{(j)}_1\ldots w^{(j)}_{\ell(j)-k}$ be the prefix of $w^{(j)}$ of length $\ell(j)-k$. Using the above-mentioned property of mixing sofic shifts we can find a sequence $(v^{(j)})_{j=1}^\infty$ of words of length $k$ such that
\[
x=u^{(1)}v^{(1)}u^{(2)}v^{(2)}\ldots u^{(n)}v^{(n)}\ldots
\]
belongs to $X$. We also have
\[
\dbar(x,w^{(1)}w^{(2)}\ldots w^{(n)}\ldots)\le k/N<\eps. \qedhere
\]
\end{proof}

The proof of the next lemma is adapted from the proof of \cite[Lemma 25]{KLO2}.

\begin{lem}\label{lem:dbar-lim-shadowing}If $X\subseteq\FS$ is a shift space such that there exists a sequence $(X_n)_{n=1}^\infty$ of shift spaces with the $\dbar$-shadowing property such that $X\subseteq X_n$ for every $n\in\N$ and $\Hdbar(X,X_n)\to 0$ as $n\to\infty$, then $X$ also has $\dbar$-shadowing.
\end{lem}
\begin{proof}
Fix $\eps>0$. Let $N\in\N$ be such that for every $n\ge N$ we have
\begin{equation}\label{ineq:one-3}
\Hdbar(X,X_n)<\eps/2.
\end{equation}
Use $\dbar$-shadowing of $X_N$ to find $M\ge 0$ such that if $y=w^{(1)}w^{(2)}w^{(3)}\ldots\in\FS$
satisfies $w^{(j)}\in\lang(X_N)$ and $|w^{(j)}|\ge M$ for every $j\in\N$ then one can find $\bar{x}\in X_N$ satisfying
\begin{equation}\label{ineq:two-3}
\dbar(\bar{x},y)<\eps/2.
\end{equation}
Since $\lang(X)\subseteq\lang(X_N)$, we can apply \eqref{ineq:one-3} and \eqref{ineq:two-3} to $y=w^{(1)}w^{(2)}w^{(3)}\ldots\in\FS$
satisfying $w^{(j)}\in\lang(X)$ and $|w^{(j)}|\ge M$ for every $j\in\N$. It follows that there is $x\in X$ such that $\dbar(x,y)<\eps$.
\end{proof}

The proof of the following lemma is adapted from the proof of \cite[Lemma 3.1]{KKO}.

\begin{lem}\label{lem:dbar-shadowing-chain-mix}
If a shift space $X\subseteq\FS$ has the $\dbar$-shadowing property and $\sigma(X)=X$, then $X$ is chain mixing.
\end{lem}
\begin{proof}
Assume that $\sigma(X)=X$. Recall that $X$ is chain mixing if and only if $X\times X$ is chain transitive. Note that it is enough to prove that $X$ is chain transitive, because if $X$ has $\dbar$-shadowing  and $\sigma(X)=X$, then the same holds for $X\times X$. 
Fix $n$. We will show that the Markov approximation $X_n$ is transitive, that is, for every $u,w\in\lang_{n-1}(X)$ there is a word $v$ such that every word of length $n$ appearing in $uvw$ belongs to $\lang_n(X)$. Fix $0<\eps<1/(4n)$. Use $\dbar$-shadowing to find $N$ for that $\eps$. Set $n_0:=\max\{N,n\}$ and $n_j=2^jn_0$ for $j\in\N$. For $j\ge 0$ let $w^{(2j)}$ be any word in $\lang(X)$ such that $|w^{(2j)}|=n_{2j}$ and $u$ is the prefix of $w^{(2j)}$ and let $w^{(2j+1)}$ be any word in $\lang(X)$ such that $|w^{(2j+1)}|=n_{2j+1}$ and $w$ is the suffix of $w^{(2j+1)}$ (the latter word exists because $\sigma(X)=X$). Thus for every $j\in\N_0$ we have words $u^{(2j)}$, $v^{(2j+1)}$ such that
\[
w^{(2j)}=uu^{(2j)}\quad\text{and}\quad w^{(2j+1)}=v^{(2j+1)}w.
\]
Let $y=w^{(0)} w^{(1)}w^{(2)}\ldots$. Write $\xi(0)=0$ and
\[\xi(k+1)=|w^{(0)}|+\ldots+|w^{(k)}|.\]
Therefore
\[
y_{[\xi(2j),\xi(2j+1))}y_{[\xi(2j+1),\xi(2j+2))}=w^{(2j)}w^{(2j+1)}=uu^{(2j)}v^{(2j+1)}w.
\]
By the $\dbar$-shadowing property there is $x\in X$ such that $\dbar(x,y)<\eps$. We claim that there exists $j\ge 0$ such that for some $\xi(2j)+n\le p <\xi(2j+1)-n$ and $\xi(2j+1) \le q <\xi(2j+2)-2n$ we have
\[
x_{[p,p+n)}= y_{[p,p+n)}\quad\text{and}\quad x_{[q,q+n)}= y_{[q,q+n)}.
\]
It follows that $y_{[\xi(2j)p)}x_{[p,q+n)}y_{[q+n,\xi(2j+2))}\in\lang(X_n)$. If the claim is false, then easy computations show that $\dbar(x,y)\ge \eps$, which would be a contradiction. This completes our proof.
\end{proof}

\begin{lem}\label{lem:dbar-shadowing-approach}
Every shift space with the $\dbar$-shadowing property is $\dbar$-approachable.
\end{lem}
\begin{proof}Let $X$ be a shift space with the $\dbar$-shadowing property. We need to show that $\Hdbar(X^M_n,X)\to 0$ as $n\to\infty$. Fix $\eps>0$. For this $\eps$ we choose $N$ using the $\dbar$-shadowing property. Let $n\ge N$.
Since $X\subseteq X^M_n$ it is enough to check that if $x\in X^M_n$ then one can find $y\in X$ with $\dbar(x,y)<\eps$. Hence we fix $x\in X^M_n$
and define the sequence $w^{(j+1)}=x_{[nj,n(j+1))}=x_{nj}x_{nj+1}\ldots x_{n(j+1)-1}$ for $j=0,1,2,\ldots$. Since for each $k\ge 1$ we have $w^{(k)}\in\lang_n(X)$, because $\lang_n(X_n)=\lang_n(X)$, we can apply $\dbar$-shadowing to find $y\in X$ such that
\[
\dbar(x,y)=\dbar(w^{(1)}w^{(2)}w^{(3)}\ldots,y)<\eps.
\]
It follows $\Hdbar(X^M_n,X)<\eps$ for $n\ge N$ as needed.
\end{proof}

We finish this section with the proof of the characterisation theorem.
\begin{proof}[Proof of Theorem \ref{thm:characterisation}.]
\eqref{cond:i}$\implies$\eqref{cond:ii} First, note that $\sigma(X)=X$ because $\sigma(X_n)=X_n$ whenever $X_n$ is a mixing sofic shift and $\sigma$ is finite-to-one.
Note that for every $n\in\N$ the  shift space $X_n$ has the $\dbar$-shadowing property by Lemma \ref{lem:spec-dbar-shadowing}. Then we apply Lemma \ref{lem:dbar-lim-shadowing}.
\eqref{cond:ii}$\implies$\eqref{cond:iii} It follows from Lemmas \ref{lem:dbar-shadowing-approach} and \ref{lem:dbar-shadowing-chain-mix}. \eqref{cond:iii}$\implies$\eqref{cond:i} It is straightforward consequence of the fact that Markov approximations for a chain mixing shift space are mixing sofic shifts.
\end{proof}

In the next section we will see that if a shift space $X \subseteq \FS$ is $\dbar$-approachable and chain transitive, then its ergodic measures are entropy-dense.

\begin{exmp}\label{exmp:non-dbar-approach}
We show that even very simple transitive sofic shift may have Markov approximations that are away in $\dbar$. This example also shows that we cannot replace mixing sofic shifts in Theorem  \ref{thm:characterisation} by transitive sofic shifts.

The shift space $X$ is defined as the set of all sequences $x\in\{0,1\}^\infty$ such that $x\in X$ if and only if $x_{2i}=0$ for every $i\ge 0$ or $x_{2i+1}=0$ for every $i\ge 0$. Clearly, $X=\tilde{X}$ is an hereditary transitive sofic shift, which is not topologically mixing. Given $k\in\N$, we define a periodic point
\[y^{(k)}=y_0^{(k)}y_1^{(k)}\ldots=((10)^{k+1} 0^{2k+1})^\infty.\]
We note that $y^{(k)}\in X^M_{2k}$ for every $k\in\N$. By definition, for every $k\in\mathbb{N}$ and $j>0$, the symbol $1$ occurs  $2j(k+1)$ times in the block $y^{(k)}_{[0,2j(4k+3))}$. These occurrences are evenly distributed among even and odd positions. Therefore for every $x\in \tilde X=X$ and $j\in\mathbb{N}$ we have
\[
\dHam(y^{(k)}_{[0,2j(4k+3)} x_{[0,2j(4k+3)})\ge \frac{j(k+1)}{j(8k+6)}\geq \frac18,\]
It follows immediately that for every $k\in\mathbb{N}$ we have
\[\dund(y^{(k)},x)\geq \frac18.
\]
Therefore for each $k\in\mathbb{N}$ it holds
\[\Hdund(X,X^M_{2k})\ge\frac18,\]
that is, $X$ is an example of  a non-$\dbar$-approachable transitive sofic shift.
In addition, since every point $y^{(k)}$ is generic for an ergodic measure $\nu^{(k)}$ on $X^M_{2k}$ we can conclude that for each $k\in\mathbb{N}$ it holds
\[\Hdbarm(\Ms(X),\Ms(X^M_{2k}))\ge \frac18.\]

It means that we can replace Markov approximation by the approximation by sofic shifts only if these sofic shift are mixing as in condition \ref{cond:i} of Theorem  \ref{thm:characterisation}.
\end{exmp}

\section{Approximation and entropy-density}
\label{sec:main}
\label{sec:main-main}

In this section, we discuss approximation schemes of shift spaces related to pseudometric $\dbar$ and metric $\dbarm$. Then we
show that these approximation schemes allow for a transfer of entropy density.

Recall that a sequence of shift spaces $(X_n)_{n=1}^\infty\subseteq\FS$ converges to a shift space $X\subseteq \FS$ in the usual hyperspace topology if $\rho^H(X_n,X)\to 0$ as $n\to\infty$, where $\rho^H$ is the Hausdorff metric corresponding to any metric $\rho$ compatible with the product topology on $\FS$. Similarly, we define the convergence in the usual hyperspace topology of $\Ms(\FS)$, and we have that $\Ms(X_n)$ converges to $\Ms(X)$ as $n\to \infty$ if
$D^H(X_n,X)\to 0$ as $n\to\infty$, where $D^H$ is the Hausdorff metric corresponding to an arbitrary metric $D$ compatible with the weak$^*$ topology on $\Ms(\FS)$
(we can replace $D$ by any other compatible metric).

We will consider four new ways in which a sequence of shift spaces $(X_n)_{n=1}^\infty\subseteq\FS$ can approximate a shift space $X\subseteq \FS$. These ways are given by the following conditions
\begin{align}
  \lim_{n\to\infty}&\Hdbarm(\Ms(X_n),\Ms(X))=0,\label{four-modes-1}\\
  \lim_{n\to\infty}&\Hdbarm(\Mse(X_n),\Mse(X))=0, \label{four-modes-2}\\
  \lim_{n\to\infty}&\Hdund(X_n,X)=0,\label{four-modes-3}\\
  \lim_{n\to\infty}&\Hdbar(X_n,X)=0. \label{four-modes-4}
\end{align}
Our aim is to show that
\[
\eqref{four-modes-4}\implies\eqref{four-modes-3}\implies(\eqref{four-modes-2}\Longleftrightarrow\eqref{four-modes-1}),
\]
and any of these modes of convergence allows us to infer entropy-density for $X$ provided that entropy-density holds for $X_n$ for all $n$.
We stress that our scheme does not require the approximation to be monotone, that is we do not require  $X_1\supseteq X_2 \supseteq \ldots$ and $X=\bigcap X_n$, although in practice it often is. Bearing in mind future applications, we also include a result (Corollary \ref{gen-scheme-summ}) dealing with monotone limits of shift spaces.

Our first main result states that $\Hdbarm$-convergence for simplices of invariant measures given by \eqref{four-modes-1}  preserves entropy-density.

\begin{thm}\label{thm:gen-scheme}
Let $(X_k)_{k=1}^\infty$ and $X$ be shift spaces over $\alf$ such that
\[\Hdbarm(\Ms(X_k),\Ms(X)) \to 0 \quad\text{ as } k \to \infty.\]
If ergodic measures are entropy-dense in $\Ms(X_k)$ for each $k \in \N$,  then ergodic measures are entropy-dense in $\Ms(X)$.
\end{thm}

A key component of the proof of Theorem \ref{thm:gen-scheme} is the fact that for every shift space $X$, we have the equality \[\Hdbarm(\Ms(X_k),\Ms(X))=\Hdbarm(\Mse(X_k),\Mse(X)),\]
which implies the equivalence $\eqref{four-modes-2}\Longleftrightarrow\eqref{four-modes-1}$.
This is established in the two following lemmas.

\begin{lem}\label{lem:pseudogeneric-dbar}
Let $Y\subseteq \FS$ be a shift space and let $\mu \in \Mse(\FS)$ be an ergodic measure.
Then 
$\dbarm(\mu,\Mse(Y))=\dbarm(\mu,\Ms(Y)).$
\end{lem}
\begin{proof}It is enough to show that for every $\mu\in\Mse(\FS)$ and $\nu \in \Ms(Y)$ there exists $\nu' \in \Mse(Y)$ such that $\dbarm(\mu,\nu')\leq \dbarm(\mu,\nu)$.
Let $\bar\lambda$ be a joining which realises the $\dbarm$ distance between $\mu\in\Mse(\FS)$ and $\nu\in\Ms(Y)$, i.e.,
\[\dbarm(\mu,\nu) = \int_{\FS\times Y} d_0(x,y) \dif {\bar\lambda(x,y)}.\]
Let $\bar\xi$ be the ergodic decomposition  of $\bar\lambda$.
We have
\[
\int_{\FS\times Y}d_0(x,y) \dif\bar\lambda(x,y)=\int_{\Msse(\FS\times Y)}\left(\int_{\FS\times Y} d_0(x,y) \dif\lambda(x,y)\right) \dif\bar\xi(\lambda).
\]
Let $E$ be the set of all $\lambda\in\Msse(\FS\times Y)$ such that
\[
 \int_{\FS\times Y}d_0(x,y) \dif \lambda(x,y)\leq \dbarm(\mu,\nu).
\]
It is clear that $E$ satisfies $\bar\xi(E)>0$. Additionally, since $\mu$ is ergodic, for $\bar\xi$-a.e.{} $\lambda\in\Msse(\FS\times Y)$ the push forward of $\lambda$ through the projection onto the first coordinate is $\mu$.
Hence, there exists $\lambda'\in\Msse(\FS\times Y)$ such that $(\pi_{\FS})_*(\lambda')=\mu$ and
\[
\int_{\FS\times Y}d_0(x,y) \dif \lambda'(x,y)\leq \dbarm(\mu,\nu).
\]
Let $\nu'=(\pi_Y)_*(\lambda')$. Then $\nu'$ is an ergodic measure supported on $Y$ and $\lambda'$ is a joining of $\mu$ and $\nu'$.
Consequently, $\dbarm(\mu,\nu')\leq \dbarm(\mu,\nu)$.
 \end{proof}


\begin{lem}
\label{lem:dm-leq-de}
\label{lem:de-leq-dm}
\label{lem:dm-eq-de}
If $X$ and $Y$ are shift spaces over $\alf$, then
\[\Hdbarm(\Ms(X),\Ms(Y)) = \Hdbarm(\Mse(X),\Mse(Y)).\]
\end{lem}
\begin{proof}
It follows from Lemma \ref{lem:pseudogeneric-dbar} that
\begin{multline}\label{ineq:long}
 \Hdbar(\Mse(X),\Mse(Y))=\\
 \max\left(
 \sup_{\mu\in\Mse(X)}\dbarm(\mu,\Ms(Y)),\ \sup_{\nu\in\Mse(Y)}\dbarm(\nu,\Ms(X)) \right).
\end{multline}
Since $\Mse(X)\subset\Ms(X)$ we have
\[
\sup_{\mu\in\Mse(X)}\dbarm(\mu,\Ms(Y))\le \sup_{\mu\in\Ms(X)}\dbarm(\mu,\Ms(Y)),
\]
and the same inequality with the roles of $X$ and $Y$ reversed.
These inequalities, together with \eqref{ineq:long} give us
\begin{multline*}
                                                      \Hdbar(\Mse(X),\Mse(Y))= \\
                                                      \max\left(
 \sup_{\mu\in\Mse(X)}\dbarm(\mu,\Ms(Y)),\ \sup_{\nu\in\Mse(Y)}\dbarm(\nu,\Ms(X)) \right) \\
 \leq \Hdbarm(\Ms(X),\Ms(Y)).
                                                    \end{multline*}
It remains to prove that for every 
$\mu\in\Ms(X)$, there exists $\nu\in\Ms(Y)$ such that $\dbarm(\mu,\nu)\leq \Hdbarm(\Mse(X),\Mse(Y))$.
This is sufficient, since the roles of $X$ and $Y$ are interchangeable meaning that
\[
\Hdbarm(\Ms(X),\Ms(Y))\le
\Hdbar(\Mse(X),\Mse(Y)).
\]
Recall that finite convex combinations of ergodic measures are weak$^*$ dense in $\Ms(X)$.
Therefore we have
\[
\mu=\lim_{n\to\infty}\sum_{j\in I(n)}\alpha\ijn\mu\ijn,
\]
where $I(n)$ is a finite set for every $n\in\N$, furthermore, for a fixed $n$ and for each $j\in I(n)$ we have $\mu\ijn\in\Mse(X)$, $\alpha\ijn>0$, and $\sum_{j\in I(n)}\alpha\ijn=1$.
By our assumption, to every $\mu\ijn$ corresponds an ergodic measure $\nu\ijn\in\Mse(Y)$ and a joining $\lambda\ijn\in J(\mu\ijn,\nu\ijn)$ with
\begin{equation}\label{ineq:dbar-bound}
\dbarm(\mu\ijn,\nu\ijn)=\int_{X\times Y}d_0(x,y)\dif\lambda\ijn(x,y)\le \dbarm(\Mse(X),\Mse(Y)).
\end{equation}
For each $n\in\N$ define
\[
\lambda\uin=\sum_{j\in I(n)}\alpha\ijn\lambda\ijn.
\]
Using compactness of $\Mss(X\times Y)$ we may assume that the sequence $(\lambda\uin)_{n\in\N}$ weak$^*$ converges to some $\bar\lambda\in\Mss(X\times Y)$.
By the definition of weak$^*$ topology, it follows immediately that
\begin{equation}\label{eq:int-lim}
\int_{X\times Y} d_0(x,y)\dif\bar\lambda(x,y)=\lim_{n\to\infty}\int_{X\times Y} d_0(x,y)\dif\lambda\uin(x,y).
\end{equation}
It is easy to see that 
$\bar\lambda$ is a joining of $\mu$ with $\nu=(\pi_Y)_*(\bar\lambda)\in\Ms(Y)$.
Definition of $\lambda\uin$ yields that for every $n\in\N$ we have
\begin{equation}\label{eq:int-lin}
\int_{X\times Y} d_0(x,y)\dif\lambda\uin(x,y)=\sum_{j\in I(n)}\alpha\ijn\int_{X\times Y} d_0(x,y) \lambda\ijn(x,y).
\end{equation}
Combining \eqref{ineq:dbar-bound}, \eqref{eq:int-lim}, and \eqref{eq:int-lin} we obtain
\[
\int_{X\times Y} d_0(x,y)\dif\bar\lambda(x,y)\le \Hdbarm(\Mse(X),\Mse(Y)).
\]
Hence, for each $\mu\in\Ms(X)$ we have $\nu\in\Ms(Y)$ such that $\dbarm(\mu,\nu)\leq \Hdbarm(\Mse(X),\Mse(Y))$.
\end{proof}

Now, having Lemma \ref{lem:dm-eq-de} at our disposal, we proceed with the proof of Theorem \ref{thm:gen-scheme}.

  \begin{proof}[Proof of Theorem \ref{thm:gen-scheme}]
   Fix $\eps>0$ and $\nu \in \Ms(X)$. Let $D$ be any metric compatible with the weak$^*$ topology on $\Ms(\FS)$. Since the entropy function $\mu\mapsto h(\mu)$ is uniformly continuous on $\Ms(\FS)$ endowed with the metric $\dbarm$ \cite{Rudolph}, we can find $\delta_1>0$ such that $|h(\mu)-h(\mu')|<\eps/3$ for any $\mu,\mu'\in\Ms(\FS)$ with  $\dbar(\mu,\mu')\leq\delta_1$.
By \cite[Theorem 7.7]{Rudolph}, there exists $\delta_2 > 0$ such that $\D(\mu,\mu')<\eps/3$
for any $\mu,\mu'\in\Ms(\FS)$ with  $\dbarm(\mu,\mu')\leq\delta_2$.
Put $\delta = \min(\delta_1,\delta_2)$ and fix $k\in\N$ such that (cf.\ Lemma \ref{lem:dm-eq-de})
\begin{equation}\label{eq:dM=dMe2}
 \Hdbarm(\Ms(X_k),\Ms(X)) = \Hdbarm(\Mse(X_k),\Mse(X)) < \delta.
 \end{equation}

By  \eqref{eq:dM=dMe2}, there exists $\mu\in\Ms(X_k)$ such that $\dbarm(\nu,\mu)< \delta$, whence
\begin{equation}\label{ineq:delta1}
D(\nu,\mu)<\eps/3\qquad\text{and}\qquad|h(\nu)-h(\mu)|<\eps/3.
\end{equation}
Since $\Mse(X_k)$ is entropy-dense in $\Ms(X_k)$, there exists $\mu'\in\Mse(X_k)$ that satisfies
\begin{equation}\label{ineq:delta2}
D(\mu,\mu')<\eps/3\qquad\text{and}\qquad |h(\mu)-h(\mu')|<\eps/3.
\end{equation}
By another application of \eqref{eq:dM=dMe2}, there exists $\nu'\in \Mse(X)$ such that $\dbarm(\nu',\mu')< \delta$, so
\begin{equation}\label{ineq:delta3}
D(\nu',\mu')<\eps/3\qquad\text{and}\qquad |h(\nu')-h(\mu')|<\eps/3.
\end{equation}
Combining \eqref{ineq:delta1}, \eqref{ineq:delta2}, and \eqref{ineq:delta3} we conclude that
\[D(\nu,\nu')<\eps\qquad\text{and}\qquad |h(\nu)-h(\nu')|<\eps.\qedhere\]
\end{proof}

Note that Theorem \ref{thm:gen-scheme}  only transfers entropy density from sequences of simplices to their $\Hdbarm$-limits. It turns out that very useful bounds for $\Hdbarm$ are provided by $\Hdund$ and $\Hdbar$. These bounds yield immediately the implications \eqref{four-modes-4}$\implies$\eqref{four-modes-3}$\implies$\eqref{four-modes-1}. This allows us to work directly with shift spaces bypassing the need of determining their simplices of invariant measures.

\begin{prop}\label{prop:dbar-le-dund}
If $X$ and $Y$ are shift spaces over $\alf$ then
\[\Hdbarm(\Mse(X),\Mse(Y))\leq \Hdund(X,Y)\le \Hdbar(X,Y).\]
\end{prop}
\begin{proof} The second inequality is obvious. We turn to the proof of the first one.
Let $\mu \in \Mse(X)$ and fix $\Delta>\Hdund(X,Y)$. Pick a $\mu$-generic point $\bar x$. By our assumptions, there is a point $\bar y\in Y$ such that $\dund(\bar x,\bar y)<\Delta$. Let $(n_k)_{k\ge 1}$ be a strictly increasing sequence of integers such that
\[
\dund(\bar x,\bar y)=\lim_{k\to\infty}\frac{1}{n_k}|\{0\le j<n_k:\bar x_j\neq\bar y_j\}|<\Delta.
\]
Passing to a subsequence if necessary, we assume that
the point $(\bar x,\bar y)\in X\times Y$ generates a measure $\bar\lambda\in\Mss(X\times Y)$, that is, the sequence of measures
\[
\bar\lambda_k=
\frac1{n_k}\sum^{n_k-1}_{j=0}\dirac_{(\sigma\times\sigma)^j(\bar x,\bar y)}\]
weak$^*$ converges to $\bar\lambda$ as $k \to \infty$.
Let $\nu=(\pi_Y)_*(\bar\lambda)\in\Ms(Y)$. Then $\bar\lambda$ is a joining of $\mu$ and $\nu$. It follows from  Lemma \ref{lem:dm-eq-de} that
\begin{align*}
\dbarm(\mu,\Mse(Y))&=\dbarm(\mu,\Ms(Y))\leq\dbarm(\mu,\nu)\leq \int_{X\times Y} d_0 (x,y)\,
                    \dif \bar\lambda(x,y)\\
  &= \lim_{k\to\infty}\int_{X\times Y} d_0(x,y)\,
\dif \bar\lambda_k(x,y) = \dund(\bar x,\bar y)<\Delta.
\end{align*}
Since $\Delta>\dund(X,Y)$ and $\mu \in \Mse(X)$ were arbitrary we have
\[
\sup_{\mu\in\Mse(X)}\dbarm(\mu,\Mse(Y))\le \Hdund(X,Y).
\]
Noting that the roles of $X$ and $Y$ are interchangeable, we see that
\begin{multline*}
  \Hdbarm(\Mse(X),\Mse(Y))=\\
\max\left(
 \sup_{\mu\in\Mse(X)}\dbarm(\mu,\Mse(Y)),\ \sup_{\nu\in\Mse(Y)}\dbarm(\nu,\Mse(X)) \right)
\le  \Hdund(X,Y). \qedhere
\end{multline*}
\end{proof}

In practice, we will not apply Theorem \ref{thm:gen-scheme} directly, instead we will invoke one of the corollaries obtained from Theorem \ref{thm:gen-scheme} and Proposition \ref{prop:dbar-le-dund}. The first one is immediate.
\begin{cor}\label{cor:gen-scheme}
Let $(X_n)_{n=1}^\infty$ be a sequence of shift spaces over $\alf$ with entropy-density such that for a shift space $X\subseteq\FS$ we have
\[
\lim_{n\to\infty}\Hdund(X_n,X)=0.
\]
Then $\Hdbarm(\Mse(X_n),\Mse(X)) \to 0$ as $n \to \infty$ and the ergodic measures are entropy-dense for $X$.
\end{cor}
Before stating the second practical corollary to Theorem \ref{thm:gen-scheme}, note that if $(X_n)_{n=1}^\infty$ is a sequence of shift spaces satisfying the Cauchy condition with respect to $\Hdbar$ or $\Hdund$, then using Proposition \ref{prop:dbar-le-dund} we get that $\Ms(X_n)_{n=1}^\infty$ satisfies the Cauchy condition for $\Hdbarm$, so the latter sequence converges in $\Hdbarm$ to some nonempty compact subset of $\Ms(\FS)$. Yet, it is not easy to identify the limit with $\Ms(X)$ for some shift space $X$, unless $X_1\supseteq X_2\supseteq\ldots$. Under this additional assumption the following result is immediately follows from Lemma \ref{gen-scheme-Cauchy} and Corollary \ref{cor:gen-scheme}.
\begin{cor}\label{gen-scheme-summ}
Let $(X_n)_{n=1}^\infty$ be a decreasing sequence of shift spaces over $\alf$ such that
\[
\sum_{n=1}^\infty\Hdund(X_n,X_{n+1})<\infty,
\]
Then $X=\bigcap_{n=1}^\infty X_n$ is a shift space such that $\Hdbarm(\Mse(X_n),\Mse(X)) \to 0$ as $n \to \infty$. Furthermore, if ergodic measures are entropy dense for $X_n$ and each $n\ge 1$, then entropy-density holds also for $X$.
\end{cor}

We also find the following observation useful. We stress that neither $\dund$ nor $\Hdund$ obey the triangle inequality, so the assumptions of Corollary \ref{cor:gen-scheme-def-Ms} do not guarantee that $\Hdund(X,Y) = 0$.

\begin{cor}\label{cor:gen-scheme-def-Ms}
  Suppose that $(X_k)_{k=1}^\infty$ and $X,Y$ are shift spaces over $\alf$ such that $\Hdund(X_k,X)$ and $\Hdund(X_k,Y)$ tend to zero as $k \to \infty$. Then
  \[\Ms(X)=\Ms(Y).
  \]
\end{cor}
\begin{proof}
It is a straightforward consequence of Proposition \ref{prop:dbar-le-dund}. 
\end{proof}

In order to apply Theorem \ref{thm:gen-scheme} or one of its Corollaries we still need to single out a class of shift spaces for which entropy-density is easily verifiable and then describe a family of
shift spaces $X$ such that for some sequence $(X_k)_{k=1}^\infty$ of shift spaces in the aforementioned class of shift spaces we have
\[\Hdbarm(\Mse(X_k),\Mse(X)) \to 0 \quad\text{ as } k \to \infty.\]

A class of shifts where entropy-density is easy to demonstrate is the family of transitive sofic shifts. The result is known even in a greater generality, see \cite{EKW, PS}. The proofs given in \cite{EKW,PS} simplify in the special case needed here (cf. proof of Theorem 7.12 in \cite{Rudolph}).
\begin{prop}\label{prop:sofic-ent-density}
Every transitive sofic shift space over a finite alphabet has entropy-dense set of ergodic measures.
\end{prop}

In applications, we will use either Corollary \ref{cor:gen-scheme} or Corollary \ref{gen-scheme-summ} together with Proposition \ref{prop:sofic-ent-density}. For further reference we formulate our observation as a proposition.

\begin{cor}\label{cor:gen-scheme-sofic-dund}
Let  $(X_k)_{k=1}^\infty$ be a sequence of transitive sofic shifts over $\alf$. If $X$ is a shift space over $\alf$ such that $\Hdund(X_k,X) \to 0$ as $k\to\infty$ then ergodic measures are entropy-dense in $\Ms(X)$.
\end{cor}

As the first application of Corollary \ref{cor:gen-scheme-sofic-dund} we may now prove entropy-density of all $\dbar$-approachable and chain transitive shift-spaces.

\begin{prop}
If a shift space $X \subseteq \FS$ is $\dbar$-approachable and chain transitive, then its ergodic measures are entropy-dense.
\end{prop}
\begin{proof}
Apply Corollary \ref{cor:gen-scheme-sofic-dund} with $X_k=X^M_k$.
\end{proof}

\section{$\dbar$-approachability and specification}\label{sec:spec}

We prove that shift spaces satisfying popular variants of the specification property (specification and almost specification) 
are $\dbar$-approachable.
Note that the definitions below are adapted to the symbolic setting. These formulations are equivalent, but slightly different than the ones
usually used for general compact dynamical systems. 
For more background on the specification property and its relatives we refer the reader to the survey paper \cite{KLO}.

A shift space $X\in \FS$ has the \emph{specification property} if there is $k\in\N$ such that for any words $u,w\in\lang(X)$ one can find a word $v$ of length $k$ such that $uvw\in\lang(X)$.
A shift space $X$ has the \emph{almost specification property} if there is a mistake function $g\colon\N\to\N$, which is a nondecreasing function with $g(n)/n\to 0$ as $n\to\infty$ and such that for any words $u,w\in\lang(X)$ there are words $u',w'$ satisfying $|u|=|u'|$, $|v|=|v'|$, $\dHam(u,u')\le g(|u|)$, and $\dHam(w,w')\le g(|w|)$, such that $u'w'\in\lang(X)$. Here and elsewhere, $\dHam$ stands for the normalised Hamming distance, that is
\[\dHam(u,w)=\frac1n|\{1\le j\le n: u_j\neq w_j\}|,
\]
where $n=|u|=|w|$.

A prototype for the almost specification property was the \emph{$g$-almost
product property} introduced in the context of general (not-necessarily symbolic) dynamical systems by Pfister and Sullivan in \cite{PS07}. Later, Thompson \cite{Thompson} proposed to slightly modify this notion and renamed it the \emph{almost specification property}. We follow Thompson, hence our almost specification property is
logically weaker (less restrictive) than the notion introduced by Pfister and Sullivan. The standard examples of shift spaces with the almost specification property are $\beta$-shifts (see \cite{PS}). For shift spaces it is easy to see that the specification property implies almost specification. For a generic $\beta>1$ the $\beta$-shift $X_\beta$ is an example of a shift space with almost specification but without specification (see \cite{Schmeling}).

Both specification properties defined above are known to imply entropy-density. For shifts with the specification property it was first proved in \cite{EKW}. The almost specification property implies entropy density because it implies the approximate product property, and the latter implies entropy-density by Theorem 2.1 of \cite{PS}.

It is also easy to see that each of the specification properties considered here implies the $\dbar$-shadowing property (actually, the proof of Lemma \ref{lem:spec-dbar-shadowing} applies almost verbatim).

\begin{lem}\label{lem:dbar-shadow-almost-spec} For a shift space $X\subseteq \FS$ specification implies almost specification.
If a shift space $X\subseteq \FS$ has the almost specification property, then it has also the $\dbar$-shadowing property.
\end{lem}

By Theorem \ref{thm:characterisation}, every shift space $X$ with the almost specification property such that $\sigma(X)=X$ is $\dbar$-approachable. Unfortunately, this the latter condition (i.e., surjectivity) is not necessarily satisifed by a shift space with almost specification. For example, the shift space $X=\{0^\infty,10^\infty\}\subseteq\{0,1\}^\infty$ has the almost specification property, but $\sigma(X)\neq X$. Therefore we need to assume $\sigma(X)=X$ in our next result.

\begin{prop}\label{prop:spec-imp-dbar-approach}
Let $X\subseteq\FS$ be a shift space with the almost specification property.
If $\sigma(X)=X$, then $X$ is $\dbar$-approachable and chain mixing.
\end{prop}
\begin{proof} By Lemma \ref{lem:dbar-shadow-almost-spec} $X$ has the $\dbar$-shadowing property. Hence $X$ satisfies condition \ref{cond:ii} of Theorem \ref{thm:characterisation}, so we conclude that $X$ is $\dbar$-approachable.
\end{proof}
\begin{rem}\label{rem:no-intrinsic-erg} One may wonder if $\dbar$-shadowing or $\dbar$-approachability implies uniqueness of the measure of maximal entropy. It follows from Proposition \ref{prop:spec-imp-dbar-approach} and examples of surjective shift spaces with the almost specification property and multiple measures of maximal entropy presented in \cite{KOR} and \cite{Pavlov} that this is not the case.
\end{rem}

We can now show that the entropy-density results from \cite{EKW,PS} are simple consequences of entropy-density of transitive sofic shifts combined with $\dbar$-approachability.

\begin{cor}
If $X\subseteq\FS$ is a shift space with the almost specification property,
then $X$ has entropy-dense set of ergodic measures.
\end{cor}
\begin{proof}Let $X^+$ be the measure-center of $X$, that is,
\[X^+=X\setminus\bigcup\{U\subseteq\FS:U\text{ open and }\mu(U)=0\text{ for every }\mu\in\Ms(X)\}.\]
Clearly, $\Ms(X)=\Ms(X^+)$. Furthermore, a shift space $X$ has the almost specification property if and only if $X^+$ does (see \cite[Theorem 6.7]{WOC} and \cite[Theorem 5.1]{KKO}). To finish the proof note that $\sigma(X^+)=X^+$ because recurrent points are dense in $X^+$ by the Poincar\'e Recurrence Theorem.
\end{proof}

\section{Entropy density of $\BB$-free shifts}
\label{sec:bfree}

In this  section, we demonstrate the utility of our approach by showing that for every $\BB$-free shift $X_\BB$, its hereditary closure $\tilde X_\BB$ has entropy-dense set of ergodic measures. Since it is quite common that $X_\BB=\tilde X_\BB$, we obtain many examples of entropy-dense $\BB$-free shifts. In fact, we show more, namely that the invariant measures on $\tilde X_\BB$ are also the space of shift-invariant measures on a potentially larger shift space $\tilde {X}^*_\BB$. 

Throughout this section we work with the alphabet $\alf = \{0,1\}$. Recall that the hereditary closure of a shift space $X\subseteq \{0,1\}^\infty$ is defined as
    \[\tilde X=\{ y \in \{0,1\}^\infty\mid\exists x \in X \  y \leq x\},\]
where the order $\leq$ on $\{0,1\}^\infty$ is defined coordinatewise, meaning that $y \leq x$ if $y_i \leq x_i$ for all $i \in \N_0$.
A shift space is \emph{hereditary} if it coincides with its own hereditary closure.
Given a set $\BB\subseteq \N$, we say that a positive integer number is $\BB$-free if it is not a multiple of any of the member of $\BB$. The set of all $\BB$-free numbers is denoted by $\FB\subseteq\N_0$, that is,
\[
\FB=\N_0\setminus\bigcup_{b\in\BB}b\N_0.
\]
We say that $\BB$ is \emph{primitive} if for each $b,b'\in \BB$ with $b\neq b'$ we have that $b$ does not divide $b'$. A set $\BB\subseteq\N$ is \emph{taut} if for every $b_0\in\BB$ we have
\[
\dund\left(\bigcup_{b\in\BB\setminus\{b_0\}}b\N_0\right) < \dund\left(\bigcup_{b\in\BB}b\N_0\right).
\]

The characteristic sequence of $\FB$ is denoted by $\eta_\BB :=1_{\FB}$, and we think of it as of an element of the full-shift $\{0,1\}^\infty$. The orbit closure $X_\BB\subseteq \{0,1\}^\infty$ of $\eta_\BB $ is called the \emph{$\BB$-free shift}. Given an enumeration of $\BB = \{b_1,b_2,\dots,\}$ with $b_i<b_j$ for $i<j$ 
and $k\in\N$ we let $\BB|k = \{b_1,\dots,b_k\}$ denote the set of the $k$ smallest elements of $\BB$ and let $\FBk \subseteq \N_0$ stand for the set $\BB|k$-free integers. Let us write $\eta_{\BB|k}\in\{0,1\}^\infty$ for the characteristic function of $\FBk$. Note that $\eta_{\BB|k}$ is a periodic point in $\{0,1\}^\infty$ hence the $\BB|k$-free shift $X_{\BB|k}$ is just the orbit of $\eta_{\BB|k}$. We clearly have $\FB\subseteq\FBk$, hence $\eta_{\BB}\le\eta_{\BB|k}$ and   $\tilde{X}_{\BB}\subseteq\tilde{X}_{\BB|k}$.
It also turns out that for every $k\in\N$ the hereditary shifts $\tilde{X}_{\BB|k}$ are transitive and sofic. This was first noticed in \cite{DKKPL} (transitivity in Proposition 3.17 and soficity in \S3.4.1). We provide an independent proof of a slightly more general fact.
\begin{prop}\label{prop:hereditary-closure-of-periodic-orbit-is-sofic} Let $x\in\{0,1\}^\infty$ be such that $\sigma^n(x)=x$ for some $n\ge 1$. If $X=\{\sigma^j(x):j=0,1,\ldots,n-1\}$ is the orbit of $x$, then its hereditary closure $\tilde{X}$ is a transitive sofic shift. Furthermore, $\tilde{X}=\{y\in\{0,1\}^\infty: y\le \sigma^j(x)\text{ for some }0\le j<n\}$. If $\sigma(x)\neq x$, then $\tilde{X}$ is not mixing.
\end{prop}
\begin{proof}
It is enough to notice that $\tilde{X}$ is presented by a labelled graph with vertices denoted $v_0,v_1,\ldots,v_{n-1}$ and edges and their labels defined as follows: for each $0\le j<n$ we put one  edge from $v_j$ to $v_{j+1\bmod n}$ labeled with $0$ and in case that $x_j=1$ we add one more edge labelled with $0$.
\end{proof}

A key fact about $\BB$-free integers which allows our argument to work is that the periodic sets $\FBk$ approximate $\FB$ with respect to the premetric $\dund$. This is a consequence of a classical result of Davenport and Erd\H{o}s.
\begin{thm}[Davenport--Erd\H{o}s, see \cite{DE1,DE2}]\label{thm:DavenportErdos}
	Let $\BB = \{b_1,b_2,\dots\} \subseteq \N$. Then \[\dund(\FBk\setminus \FB) \to 0 \quad\text{as}\quad k \to \infty.\]
\end{thm}

Observing that
$\dund(\FBk\setminus \FB) =\dund(\eta_{\BB|k},\eta_\BB)$ we see that the characteristic function $\eta_\BB$ is the $\dund$-limit of a sequence of periodic points $\eta_{\BB|k}$. Furthermore, we can reformulate Davenport--Erd\H{o}s theorem in terms of the premetric $\Hdund$.
\begin{cor}\label{cor:DE-for-shifts}
Let $\BB = \{b_1,b_2,\dots\} \subseteq \N$. Then \[\Hdund(\tilde{X}_{\BB|k},\tilde{X}_{\BB}) \to 0 \quad\text{as}\quad k \to \infty.\]
\end{cor}
\begin{proof} Fix $k\in\N$. We claim that $\Hdund(\tilde{X}_{\BB|k},\tilde{X}_{\BB}) \le \dund(\eta_{\BB|k},\eta_\BB)$. Since $\tilde{X}_{\BB}\subseteq\tilde{X}_{\BB|k}$ it is enough to show that for every $x\in\tilde{X}_{\BB|k}$ there is $y\in\tilde{X}_{\BB}$ satisfying $\dund(x,y)\le \dund(\eta_{\BB|k},\eta_\BB)$. To this end, take $x\in\tilde{X}_{\BB|k}$.
By Proposition \ref{prop:hereditary-closure-of-periodic-orbit-is-sofic} there exists $m \geq 0$ such that $\sigma^m(\eta_{\BB|k} )\geq x$ coordinatewise. Consider $y\in \{0,1\}^\infty$ defined by the formula
\[
	y_i =
	\begin{cases}
	x_i & \text{ if } (\eta_{\BB|k})_{i+m}=(\eta_\BB)_{i+m}, \\
	0  & \text{ otherwise, i.e.{} if } (\eta_{\BB|k})_{i+m}=1 \neq 0 = (\eta_\BB)_{i+m}.
	\end{cases}
\]	
We immediately see that $y\le \sigma^m(\eta_{\BB} )$ coordinatewise, hence $y\in\tilde{X}_{\BB}$.
It is also clear that $\dund(x,y)\le \dund(\eta_{\BB|k},\eta_\BB)=\dund(\FBk\setminus \FB)$. We use Davenport--Erd\H{o}s theorem to conclude that $\Hdund(\tilde{X}_{\BB|k},\tilde{X}_{\BB}) \to 0$ as $k \to \infty$.
\end{proof}



Before we state our main result regarding $\BB$-free shifts, let us first discuss some technical issues caused by the fact that the ``$\Hdund$-approximation'' appearing in \eqref{four-modes-3} and Corollary \ref{cor:DE-for-shifts} does not uniquely determine its ``limit''.

\begin{rem}
Observe that we can consider a whole spectrum of intermediate shift spaces associated to $\BB$. Namely, for $k \in \N$ we put
  \begin{align*}
   \tilde X^{(k)}_\BB=\bigcap_{\substack{\BB'\subseteq\BB\\ \#\BB'=k}} \tilde X_{\BB'}.
  \end{align*}
These sets are again shift-invariant and hereditary. Moreover, setting
\[\tilde X^*_\BB=\bigcap^\infty_{k=1}\tilde X^{(k)}_{\BB}\]
we have the sequence of inclusions
\begin{equation}\label{eq:hierarchy}
 X_\BB\subseteq \tilde X_\BB\subseteq \tilde X^*_\BB \subseteq\ldots\subseteq \tilde X^{(2)}_\BB\subseteq \tilde X^{(1)}_\BB.
\end{equation}
We will see later (see Remark \ref{rem:hierarchy}) that the second inclusion in \eqref{eq:hierarchy} may be strict. As far as we know, only the sets $X_\BB$, $\tilde X_\BB$, and $\tilde X^{(1)}_\BB$ have appeared in the literature before. The shift space $\tilde X^{(1)}_\BB$ is the largest shift-space related to $\BB$-free constructions. It is usually called the \emph{$\BB$-admissible shift} and its elements are  \emph{$\BB$-admissible sequences} (see \cite{DKKPL}). Let us also point out that the hierarchy introduced above is still rather rough: given any countable family $\mathcal C$ of subsets of $\BB$, the intersection $\bigcap_{\BB'\in\mathcal C} \tilde X_{\BB'}$ is a hereditary shift space that contains $X_{\BB}$.
\end{rem}

\begin{lem}\label{lem:hierarchy} For $\BB\subseteq\N$ we have
\[\tilde{X}^*_{\BB}=\bigcap_{\substack{\BB'\subseteq\BB\\ \BB'\text{ finite}}}\tilde X_{\BB'}=\bigcap^\infty_{k=1}\tilde X_{\BB|k}.
\]
\end{lem}
\begin{proof}
Fix $n \in \N$. We have
\[
\bigcap_{\substack{\BB'\subseteq\BB\\ \BB'\text{ finite}}}\tilde X_{\BB'}\subseteq\tilde{X}^*_{\BB}=\bigcap^\infty_{k=1}\tilde X^{(k)}_{\BB}\subseteq \tilde X^{(n)}_{\BB}=\bigcap_{\substack{\BB'\subseteq\BB\\ \#\BB'=n}} \tilde X_{\BB'}
\subseteq \tilde{X}_{\BB|n}.\]
It follows that
\[
\bigcap_{\substack{\BB'\subseteq\BB\\ \BB'\text{ finite}}}\tilde X_{\BB'}\subseteq \tilde{X}^*_{\BB}\subseteq\bigcap^\infty_{k=1}\tilde X_{\BB|k}.
\]
Assume that $x\in \tilde X_{\BB|k}$ for every $k\in\mathbb{N}$. Take any finite set $\BB'\subseteq\BB$.
Then there exists $n\in\mathbb{N}$ such that $\BB'\subseteq(\BB|n)$, hence
$\tilde X_{\BB|n}\subseteq\tilde X_{\BB'}$. We conclude that $x\in \tilde X_{\BB'}$ for every  finite set $\BB'\subseteq\BB$. Therefore
\[
\bigcap^\infty_{k=1}\tilde X_{\BB|k}\subseteq\bigcap_{\substack{\BB'\subseteq\BB\\ \BB'\text{ finite}}}\tilde X_{\BB'},
\]
which finishes the proof.
\end{proof}


\begin{rem}\label{rem:hierarchy}
All the hierarchy \eqref{eq:hierarchy} collapses to a single shift-space ($X_\BB=\tilde X^{(1)}_\BB$) whenever $\BB$ is a taut set containing an infinite set of pairwise coprime integers (see \cite[Theorem B]{DKKPL} and \cite[Corollary 2]{Keller}, cf. proof of Corollary \ref{cor:B-free} below). However, in general, each inclusion in \eqref{eq:hierarchy} can be strict. For the example showing that the first inclusion may be strict see \cite{DKKPL}. We show that $\tilde X^{*}_\BB$ can differ from the $\BB$-admissible shift $\tilde X^{(1)}_\BB$. Indeed, if $\BB=\{4,6\}$, then $\tilde X^{*}_\BB$ and $\tilde X^{(1)}_\BB$ differ. More generally, if $\BB = \{ q p_1, q p_2, \dots, q p_k \}$ where $q \geq k$ and $p_1,\dots,p_k$ are distinct primes then  $\tilde X^{*}_\BB = \tilde X^{(k)}_\BB$ and $\tilde X^{(k-1)}_\BB$ differ.

The following example shows that $\tilde X_\BB$ and $\tilde X^{*}_\BB$ can be different as well.
\begin{exmp}
	Let $p$ be a large prime number (for instance $p = 107$), put $l = (p-2)(p-1)$, and $\BB = \{p-2, p-1,p \} \cup \{n \geq l \ : \ p \nmid n\}$. Consider the characteristic sequence $x = 1_A \in \{0,1\}^\infty$ of the set $A \subseteq \N_0$ given by
	\[
		A = \{j\in\N_0:0\le j <l\} \setminus \left( (p-2) \N \cup (p-1) \N \cup (p \N - 1) \right).
	\]
Then $x \in \tilde X^{*}_\BB$, because for any finite $\BB' \subseteq \BB$ we can find $n \in \N$ with $n \equiv 1 \bmod{p}$ and $n \equiv 0 \bmod{b}$ for all $b \in \BB' \setminus \{p\}$.
	
	On the other hand, we claim that $x \notin \tilde{X}_\BB$. For the sake of contradiction, suppose that $x  \in \tilde{X}_\BB$. Since $\tilde{X}_\BB$ is finite, there exists $m \in \N_0$ such that $x \leq \sigma^m(\eta_{\BB})$ coordinatewise. Since $n \not \in \FB$ for $n \geq l$ and $l - 1 \in A$, it follows that $m = 0$. However, $p \in A$ while $p \not \in \FB$ so $m \neq 0$, which is the sought contradiction. 	
\end{exmp}
\end{rem}


\begin{thm}\label{thm:entropy-dense-Bfree}
  Let $\BB\subseteq\N$. Then ergodic measures are entropy-dense for $\tilde{X}_\BB$ and
  \[\Ms(\tilde{X}_\BB)=\Ms(\tilde X^{*}_\BB)=\bigcap_{k=1}^\infty\Ms(\tilde{X}_{\BB|k}).\]
  \end{thm}
\begin{proof}
The second equality is a consequence of Proposition \ref{prop:intersection} and Lemma \ref{lem:hierarchy}.
Fix $k \in \N$. We have
\[\tilde{X}_{\BB}\subseteq\tilde{X}^*_{\BB}
\subseteq \tilde{X}_{\BB|k}.\] It follows that we have 
\[
\dund(\tilde{X}_{\BB|k},\tilde{X}^*_{\BB}) \leq \dund(\tilde{X}_{\BB|k},\tilde{X}_{\BB}).
\]
In particular, by Corollary \ref{cor:DE-for-shifts} all ``distances'' above tend to $0$ as $k \to \infty$ and $\Ms(\tilde{X}_{\BB})=\Ms(\tilde{X}^*_{\BB})$ 
by Corollary \ref{cor:gen-scheme-def-Ms}. We finish the proof applying Corollary \ref{cor:gen-scheme-sofic-dund}.
\end{proof}

\begin{exmp}\label{exmp:bfree-non-dbar-approach}
Note that for $\tilde X_\BB$ we have only a $\Hdund$-approximation by transitive sofic shifts, which are not mixing. It raises a question, whether $X_\BB$ or $\tilde X_\BB$ can be $\dbar$-approachable. Note that the shift space in Example \ref{exmp:non-dbar-approach} can be equivalently defined as $\tilde X_\BB$, where $\BB=\{2\}$ showing that hereditary closures of some $\BB$-free shifts are non-$\dbar$-approachable . We will extended this example and prove that there exists an infinite set $\BB$ such that $\tilde X_\BB$ is not $\dbar$-approachable. Let $b_1=2$ and pick $b_2,b_3,\ldots$ such that
\begin{equation}\label{eq:one-eight}
  \sum^\infty_{k=2}\frac1{b_k}< 1/32.
\end{equation}
Note that \eqref{eq:one-eight} ensures that $\Hdbar(\tilde X_{\{2\}},\tilde X_\BB)< 1/32$. In particular, for $n$ large enough, there is a word $u\in \lang(\tilde X_\BB)$ with $|u|=2n$ such that $\dHam((01)^n,u)<1/4$. In particular, at least half among the even entries in $u$ are occupied by $1$. In other words,
\[
|\{1\le j\le n: u_{2j}=1\}|\ge \frac12 n.
\]
By heredity of $\tilde X_\BB$, the periodic point $y=(u0^{2n-1})^\infty$ belongs to $(\tilde X_\BB)^M_{2n}$. Reasoning as in Example \ref{exmp:non-dbar-approach} we obtain
$\dbar(y,x)\geq 1/16$ for every $x\in \tilde X_{\{2\}}$. It implies that $\Hdbar((\tilde X_\BB)^M_{2n},\tilde X_{\{2\}})\ge 1/16$ yielding
\[\Hdbar((\tilde X_\BB)^M_{2n},\tilde X_\BB)> 1/32,\]
  for $n$ large enough. So the Markov approximations are $\dbar$-far from $\tilde X_\BB$. Note that we can choose $\BB$ that consists of pairwise prime numbers and still satisfies \eqref{eq:one-eight}. For such a set $\BB$ we have $\tilde X_\BB=X_\BB$ and the $\BB$-free shift  $X_\BB$ is not $\dbar$-approachable.
\end{exmp}

So far we have discussed only the hereditary closure $\tilde X_\BB$ of $X_\BB$. In general, $X_\BB$ does not need to be hereditary, as shown by the example $\BB = \N_{\geq 3}$ where $X_\BB$ is the (finite) orbit of the point $011000\dots$ and hence does not contain the point $01000\ldots$. Nevertheless, for many examples of the sets $\FB$ studied in the literature the associated $\BB$-free shift turns out to be hereditary (e.g. the square-free shift or shifts associated with abundant numbers, see \cite{DKKPL}). For the record, we note a version of Theorem \ref{thm:entropy-dense-Bfree} for these shifts.
But first, note that every hereditary shift contains the fixed point $0^\infty$, and since every $\BB$-free shift has a unique minimal subsystem (by Theorem A in \cite{DKKPL}), the set $\{0^\infty\}$ must be the unique minimal subset for a hereditary $\BB$-free shift. In other words, an hereditary $\BB$- free shift must be proximal. For $\BB$-free shifts proximality is equivalent to the fact that the set $\BB$ contains an infinite pairwise co-prime subset. Therefore hoping for $X_\BB=\tilde X_\BB$ we have to assume the latter condition.
\begin{cor}\label{cor:B-free}If $\BB$ contains an infinite sequence of pairwise co-prime integers, then ergodic measures are entropy-dense in $\Ms(X_\BB)$.
\end{cor}
\begin{proof} By Theorem C in \cite{DKKPL}, for every $\BB\subseteq\N$ there is a unique taut $\BB'\subseteq \N$ such that $\tilde X_{\BB'}\subseteq \tilde X_\BB$ and $\Ms(\tilde X_{\BB'})=\Ms( \tilde X_\BB)$. Therefore we may assume that $\BB$ is taut and proximal. Then we use a recent result of Keller \cite[Corollary 2]{Keller}, who complemented some results from \cite{DKKPL} by showing that if $\BB$ is taut and $X_\BB$ is proximal, then $X_\BB$ is hereditary, that is $X_\BB=\tilde X_\BB$. We finish the proof by applying Theorem \ref{thm:entropy-dense-Bfree}.
\end{proof}
\begin{rem}
Note that $\Ms(X_\BB)$ may be trivial, that is, its unique element may be the Dirac measure concentrated on the fixed point $0^\infty$. This is the case when $\BB$ is \emph{Behrend} (see \cite[p. 5437]{DKKPL}).
\end{rem}

\begin{rem}\label{rem:more-than-B-free}
Note that Theorem \ref{thm:entropy-dense-Bfree} remains true if we replace $\tilde{X}_{\BB}$ by any shift space which is the hereditary closure of the orbit closure of a characteristic function of a set $A\subseteq \N_0$ such that there exists a sequence of periodic sets $(A_k)_{k=1}^\infty$ satisfying
\begin{gather*}
A\subseteq A_k\quad\text{for every }k\in\N,\\
\dund(A_k,A) \to 0\quad (k \to \infty).
\end{gather*}
Such shift spaces have been already considered in the literature \cite{KKL,KLO2}. A closely related notion of a \emph{rational set} (a subset of $\N$ whose characteristic function can be $\dbar$-approximated by periodic characteristic functions) also attracted attention, see \cite{BR,BKPLR,Deka}.
\end{rem}



\section{Inner $\dbar$-approachability}\label{sec:inner}
So far we have considered shift spaces, which are approximated from the outside, that is, approximating subshifts are not contained in the approximated shift space. Inspired by an idea of approximation  given by Dan Thompson in his unpublished manuscript \cite{Thompson-note} we are going to show that every topologically mixing $S$-gap shift is
approximated in the $\Hdbar$ sense by a sequence of shift spaces of finite type contained in $X_S$. We call this property
\emph{inner $\dbar$-approachability}.

We say that a shift space $X\subseteq\FS$ is \emph{inner $\dbar$-approachable} if it contains a sequence $(X_n)_{n=1}^\infty$ of mixing shifts of finite type such that
\[
\overline{\bigcup_{n=1}^\infty X_n}=X\quad\text{and}\quad \lim_{n\to\infty}\Hdbar(X_n,X)=0.
\]

Given $S\subseteq\N$, we write $S=\{n_1,n_2,\ldots\}$ with $n_i<n_{i+1}$ for $i < |S|$, where $|S|$ stands for the cardinality of $S$ ($|S|=\infty$ if $S$ is infinite).
The $S$-gap shift $X_S$ is a shift space over $\{0,1\}$ consisting of all sequences such that the number of 0s between
any two successive occurrences of the symbol 1 belongs to $S$. Equivalently, $\{10^n1:n\notin S\}$ is the collection of forbidden sequences for $X_S$.
By \cite[Example 3.4]{Jung}, $X_S$ is topologically mixing if and only if $\gcd \{n+1:n\in S\}=1$, and topologically mixing $X_S$ has
the specification property if and only if $\sup_i |n_i - n_{i+1}|< \infty$. Note that all sequences in $X_S$ are labels of infinite paths in of the labelled directed graph $G_S = (V,E_S)$, where $V=\{v_j:0\le j\le |S|\}$ is the set of vertices and there is an edge $v_i\to v_j$ in $E_S$ if and only if $j =i+1$ or $i \in S$ and
$j = 0$. We label each edge $v_i\to v_{i+1}$ with 0 and all other edges with 1.

\begin{prop}\label{prop:inner-dbar-S-gap} Every mixing $S$-gap shift is inner $\dbar$-approachable.
\end{prop}
\begin{proof}
Assume that $S=\{s_1,s_2,\ldots\}$ is infinite, $s_i<s_{i+1}$ for every $i\in\mathbb{N}$,  and $\gcd \{n+1:n\in S\}=1$ (if $S$ is finite, then $X_S$ is a shift of finite type).
Let $N$ be such that $\gcd \{s_1+1,\ldots,s_N+1\}=1$ .
Let $L>0$ be the smallest integer such that for every $\ell\ge L$ there are nonnegative integers $\alpha_1,\ldots,\alpha_N$ such that
\[
\alpha_1(s_1+1)+\ldots+\alpha_N(s_N+1)=\ell.
\]
It follows that for every $\ell\ge L$ there are $r\in\N$ and $q_1,\ldots,q_r\in\{s_1,\ldots,s_N\}$ such that the word
\begin{equation}\label{eqn:w}
w=0^{q_1}10^{q_2}1\ldots 0^{q_r}1
\end{equation}
satisfies $w\in\lang(X_S)$ and $|w|=\ell$.
Set $S[n]=\{s_1,\ldots,s_n\}$ and let $X_n$ be the $S$-gap shift generated by $S[n]$. Clearly $X_1\subseteq X_2\subseteq X_3\subseteq\ldots$ and
\[
\overline{\bigcup_{n=1}^\infty X_n}=X_S.
\]
Furthermore, for every $n\ge N$ the shift space $X_n$ is a mixing shift of finite type. Fix $\eps>0$. Let $M\ge N$ be such that $1/s_M<\eps/2$. Let $K\ge M$ satisfy $(L+s_M)/s_K<\eps/2$. 
We claim that for every $k\ge K$ we have
\[\Hdbar(X_k,X)\le \eps.\]
Note that our proof is finished once we show that the claim holds.
For a proof of our claim fix $k\ge K$. Let $x\in X$. Since $\dbar$ is $\sigma$-invariant we assume that
\[
x=10^{t_1}10^{t_2}1\ldots 10^{t_j}1\ldots,
\]
that is, we assume that $x$ begins with 1 and contains infinitely many 1s. The proof is similar, if the symbol 1 occurs in $x$ only finitely many times.
Note that $x\notin X_k$ means that for some $j\in\N$ we have $t_j\in S\setminus\{s_1,\ldots,s_k\}$.

We will construct, for each $j$ such that $t_j \in S \setminus \{s_1,\ldots,s_k\}$, a word $v_j$ with $|v_j| = t_j$ such that $\dHam(v_j, 0^{t_j}) < \eps$ and $1v_j1 \in \lang(X_S[k])$. Put also $v_j = 0^{t_j}$ if $t_j \in \{s_1,\ldots,s_k\}$. Once this is accomplished, we set $x' = 1 v_1 1 v_2 1 \dots$.
It is now enough to observe that $x' \in X_{S[k]}$ and $\dbar(x,x') \leq \sup_j \dHam(v_j, 0^{t_j})$.


Fix $j$ such that $t_j \in S \setminus \{s_1,\ldots,s_k\}$. Let $\ell$ be the unique integer such that $L\le\ell <L+s_M$ and $t_j+1-\ell$ is divisible by $(s_M+1)$. We replace the suffix $0^{\ell-1}1$ of $10^{t_j}1$ by a word $w$ with $|w|=\ell$ having the form as in \eqref{eqn:w}. 
The remaining prefix of $ 10^{t_j}1$ now has the form
\[
10^{t_j+1-\ell}
\]
where $t_j+1-\ell=p(s_M+1)$ for some $p\ge 1$. Therefore we may replace $0^{t_j+1-\ell}$ by
\[
(0^{s_M}1)^p.
\]
We have found a word $v$ with $|v|=t_j$ such that
\[
u=1v1=1(0^{s_M}1)^p0^{q_1}10^{q_2}1\ldots 0^{q_r}1\in\lang(X_S[k]).
\]
Note that
\[
\dHam(0^{t_j},v)\le \frac{1}{t_j}(t_j/(s_M+1)+(L+s_M))<\eps,
\]
which completes the proof.
\end{proof}

Entropy-density for mixing $S$-gap shifts follows from Theorem B and Theorem 3.5 in \cite{CTY}.
Using Proposition \ref{prop:inner-dbar-S-gap} and Corollary \ref{cor:gen-scheme} we obtain a new proof of this fact.  

\begin{cor}
Every mixing $S$-gap shift has entropy dense set of ergodic measures.
\end{cor}

\section*{Acknowledgments}
We thank the referees for their positive comments and corrections that helped us to improve this paper.   
We are grateful to Dan Thompson for sharing his unpublished manuscript \cite{Thompson-note}. We would like to thank Aurelia Dymek for reading the preprint and sharing with us her helpful and insightful remarks. D.~Kwietniak was supported by
the National Science Centre (NCN) Opus grant no. 2018/29/B/ST1/01340. J.~Konieczny is working within the framework of the LABEX MILYON (ANR-10-LABX-0070) of Universit\'{e} de Lyon, within the program ``Investissements d'Avenir'' (ANR-11-IDEX-0007) operated by the French National Research Agency (ANR). He also acknowledges support from the Foundation for Polish Science (FNP).

\end{document}